\documentclass[12pt,a4paper]{article}
\usepackage{array}
\usepackage{longtable}
\usepackage{textcomp}
\usepackage{latexsym}
\usepackage{varioref}
\usepackage{xspace}
\usepackage{amsmath}
\usepackage{amsfonts}
\usepackage{makeidx}
\usepackage{verbatim}
\usepackage{graphicx}
\usepackage{tabularx}
\usepackage{mathrsfs}
\usepackage{xcolor}
\usepackage{epstopdf}
\graphicspath{{graph/}}
\newtheorem{thm}{Theorem}

\newtheorem{pro}{Proposition}

\newcommand{\calE}{{\cal E}}
\newcommand{\calS}{{\cal S}}
\newcommand{\proof}{\noindent {\bf Proof} \quad}
\newcommand{\eproof}{$\quad \Box$}

\newtheorem{Remark}{Remark}

\title{Double Well Potential Function and Its Optimization in
The n-dimensional Real Space -- Part I}
\author{ {\small  Shu-Cherng Fang} \\
{\it\small Department of Industrial and Systems Engineering, North Carolina State University, USA}\\
{\small David Yang Gao} \\
{\it \small School of Science, Information Technology, and Engineering, University of Ballarat, Australia}\\
{\small  Gang-Xuan Lin}\\
{\small  Ruey-Lin Sheu} \\
{\it\small Department of Mathematics, National Cheng Kung University, Taiwan}\\
{\small  Wen-Xun Xing}\\
{\it\small  Department of Mathematical Sciences, Tsinghua University, P. R. China}}

\begin{document}\date{}

\maketitle

\begin{abstract}
A special type of multi-variate polynomial of degree 4, called the
double well potential function, is studied. When the function is
bounded from below, it has a very unique property that two or more
local minimum solutions are separated by one local maximum solution,
or one saddle point. Our intension in this paper is to categorize
all possible configurations of the double well potential functions
mathematically. In part I, we begin the study with deriving the
double well potential function from a numerical estimation of the
generalized Ginzburg-Landau functional. Then, we solve the global
minimum solution from the dual side by introducing a geometrically
nonlinear measure which is a type of Cauchy-Green strain. We show
that the dual of the dual problem is a linearly constrained convex
minimization problem, which is mapped equivalently to a portion of
the original double well problem subject to additional linear
constraints. Numerical examples are provided to illustrate the
important features of the problem and the mapping in between.

%
%
%, we propose to look into a general model containing a discrete
%version of the generalized Ginzburg-Landau system. We are interested
%in finding all extremal points, including the local optimizers,
%global optimizers and saddle points. We show that the double well
%potential problem can be transformed into a (non-convex) quadratic
%programming with one non-convex  but positive semi-definite
%quadratic constraint. By the positive semi-definiteness of the
%constraint, we decompose the space to attain the positive definite
%part of the constraint. After space reduction, the objective
%function and the constraint of the subproblem satisfy simultaneously
%diagonalize via congruence, which have many analytic results. We
%propose the solution for the double well potential problem by
%solving the subproblem via the canonical duality technique. Since
%this research is the first ever mathematical programming point of
%view to analyze the generalized Ginzburg-Landau system, we believe
%that the results must be novel, prominent and practical.

\end{abstract}
{\bf Key Words:} Non-convex quadratic programming, Polynomial
optimization, Generalized Ginzburg-Landau functional, Double well
potential, Canonical duality.

\section{Introduction}\label{sec0}
In this paper, we propose a model that minimizes a special type of
multi-variate polynomial of degree 4 in the following form:
%The double well potential problem that we propose is to minimize
%globally the following :
\begin{eqnarray}\label{DWP}
(\mathrm{DWP}):\ \ \  \min
    \left\{ \frac{1}{2}\left(\frac{1}{2} \|Bx-c\|^2-d \right)^2+\frac{1}{2}x^TAx-f^Tx  \; | \;\; x \in R^n \right\} ,
\end{eqnarray}
where $A$ is an $n\times n$ real symmetric matrix, $B\not=0$ is an
$m\times n$ real matrix, $c\in R^m$, $d \in R$,  and $f \in R^n$.
Typical examples with properly selected parameters of the objective
function are shown in Figure 1. The left most picture in Figure 1
is the simplest example with $n=1$ where there are two local energy
wells separated by one barrier. A higher dimensional analogy is
shown in the center picture of Figure 1. Note that the barrier in
this case is not a local maximum but a saddle point. The figure in
the right most is called the Mexican hat potential. It is created by
selecting a negative definite matrix $B$ and setting $A=0, f=0$ and
$ c=0$. It forms a ring-shaped region of infinitely many global
minima with one unique local maximum sitting in the center. Due to
the common feature shown in these illustrative examples, the
objective function is called a {\it double well potential function}
and the (DWP) model is referred to as the {\it double well potential
problem}.

\begin{figure}[!hbtp]
\begin{center}
\includegraphics[width=1.0\textwidth]{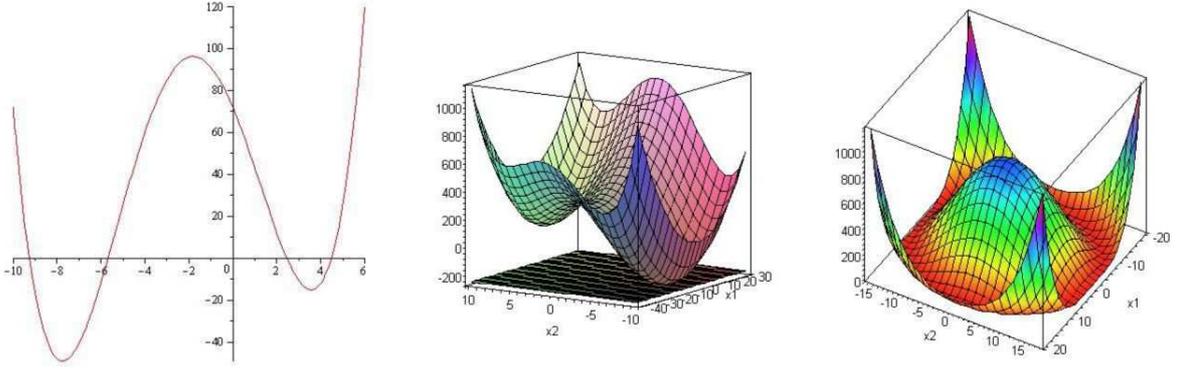}
\end{center}
\caption{Illustrative examples for the double well potential functions (DWP).}
\end{figure}

One motivation to investigate the (DWP) problem came from
  numerical approximations to the generalized Ginzburg-Landau
functionals \cite{Gao3}. The functionals often describe the total energy of a
ferroelectric system such as the ion-molecule reactions
\cite{brauman}. In a ferromagnetic spin system, the critical
phenomena and the phase transition is studied by the mean field
approach which also involves a double well potential
\cite{bodineau,kaski}. Other applications of the Ginzburg-Landau
functionals can be found in solid mechanics and quantum mechanics
\cite{Gao3, heuer}.

The mathematical formula of the generalized Ginzburg-Landau
functionals   takes the following general form
\cite{Gao3, Jerrard}:
\begin{eqnarray}\label{int}
I^{\alpha}(\mu)=\int_\Omega[\frac{1}{n}\|\nabla\mu(x)\|^n+\frac{\alpha}{2}(\frac{1}{2}
\|\mu(x)\|^2-\beta)^2]dx,
\end{eqnarray}
where $\Omega\subset R^n$, $\alpha,\beta$ are positive material
constants, and $\mu:\Omega\longrightarrow R^q$ is a smooth
vector-valued (field) function describing the phase (order) of the
system. It is known that, when $\alpha$ is sufficiently large (so
that the second term dominates), if the trace of $\mu$ on the
boundary $\partial \Omega$ is a function of non-zero Brouwer degree,
then the generalized Ginzburg-Landau functional $I^{\alpha}(\mu)$ is
bounded from below by $\ln \alpha$ \cite{Jerrard}. The second term
of (\ref{int}) is  actually the  double-well potential in the
integral form. Directly minimizing $I^{\alpha}(\mu)$ over any
reasonable functional space is, in general, very difficult. Hence
only the lower bound is estimated in the literature. We therefore
look into the discrete version of (\ref{int}) and it naturally leads
to a
special case of (\ref{DWP}). %Moreover, not only the global minimizer
%of $I^{\alpha}(\mu)$, but also any critical point of the system has
%a meaning in the applications. We are thus interested in all
%extremal points, including the local/global minimizers/maximizers
%and the saddle points of (\ref{DWP}). This is the reason we purpose
%to ``optimize'' $P(x)$, not just to ``minimize'' it. A thorough
%investigation into (\ref{DWP}) will lead to a complete understanding
%of the generalized Ginzburg-Landau system and the double-well
%potential as well.
%
%
%
%\section{Preliminary results of the proposed research}\label{sec0-1}

To illustrate how (\ref{int}) can be discretized into (\ref{DWP}),
we work out an example with $n=2$, $q=1$, and
$\Omega=\Omega_x\times\Omega_y=[0,1]\times[0,1]$. Let
$\{0=x_1<x_2<\ldots<x_{s+1}=1\}$ be a partition of $\Omega_x$ that
divides $[0,1]$ into $s$ subintervals of equal length. Similarly,
let $\{0=y_1<y_2<\ldots<y_{t+1}=1\}$ be the uniform grid of
$\Omega_y$. Define an $(s+1)\times(t+1)$ vector $e$ by
$$e=[e_{1,1},e_{2,1},\cdots,e_{s+1,1},
e_{1,2},e_{2,2},\cdots,e_{s+1,2},
\cdots,e_{1,t+1},e_{2,t+1},\cdots,e_{s+1,t+1}]^T,$$ where
$e_{i,j}=\mu(x_i,y_j)$. With the partition, we can approximate
$\nabla \mu$ by the first order difference and approximate
$I^{\alpha}(\mu)$ by the Riemann sum so that a discrete version of
(\ref{int}) becomes
\begin{align}
%\begin{eqnarray}\label{Riemann Sum}
%\begin{array}{ll}
%&
&\sum\limits_{i=1}^s\sum\limits_{j=1}^t\frac{1}{2}|(\frac{e_{i+1,j}-e_{i,j}}
{\frac{1}{s}})^2+(\frac{e_{i,j+1}-e_{i,j}}{\frac{1}{t}})^2|
\cdot\frac{1}{s}\frac{1}{t}+\sum\limits_{i=1}^s\sum\limits_{j=1}^t\frac{\alpha}{2}
(\frac{1}{2}e_{i,j}^2-\beta)^2\cdot\frac{1}{t}\frac{1}{s}\notag\\
%=&
&=\sum\limits_{i=1}^s\sum\limits_{j=1}^t\frac{s}{2t}(e_{i+1,j}-e_{i,j})^2+
\sum\limits_{i=1}^s\sum\limits_{j=1}^t\frac{t}{2s}(e_{i,j+1}-e_{i,j})^2+
\sum\limits_{i=1}^s\sum\limits_{j=1}^t\frac{\alpha}{2st}
(\frac{1}{2}e_{i,j}^2-\beta)^2\notag\\
%=&
&\hspace*{-2pt}\begin{array}{ll}
=&\displaystyle{\sum\limits_{i=1}^s\sum\limits_{j=1}^t}\frac{s}{2t}
(e_{i,j},e_{i+1,j})E(e_{i,j},e_{i+1,j})^T
+\sum\limits_{i=1}^s\sum\limits_{j=1}^t\frac{t}{2s}
(e_{i,j},e_{i,j+1})E(e_{i,j},e_{i+1,j})^T\\
&+\displaystyle{\sum\limits_{i=1}^s\sum\limits_{j=1}^t}\frac{\alpha}{2st}
(\frac{1}{2}e_{i,j}^2-\beta)^2\end{array}\label{Riemann Sum}
%\end{array}
%\end{eqnarray}
\end{align}
where
$E=\tiny{\left[\begin{array}{cc}1&-1\\-1&1\end{array}\right]}$.
The sum of quadratic forms in (\ref{Riemann Sum}) can be further
combined into a large quadratic form. Let
$\mathcal{B}_i=Diag(0_{i-1},E,0_{(s+1)(t+1)-i-1})$, where
$0_k$ is a $k\times k$ block matrix of $0$. Then,
\begin{eqnarray}\label{L1}
\begin{array}{l}
\sum\limits_{i=1}^s\sum\limits_{j=1}^t\frac{s}{2t}
(e_{i,j},e_{i+1,j})E(e_{i,j},e_{i+1,j})^T
=\frac{1}{2}e^T(\frac{s}{t}\sum\limits_{i\in \mathcal{T}}\mathcal{B}_i)e
\end{array}
\end{eqnarray}
where
$\mathcal{T}=\{1,2,3,\cdots,(s+1)t\}\backslash\{(s+1),2(s+1),3(s+1),\cdots,t(s+1)\}$.
Analogously, we can define $\mathcal{C}_i$ to be an
$(s+1)(t+1)\times(s+1)(t+1)$ matrix with the $(i,i)$ and
$(i+(s+1),i+(s+1))$ components being $1$, $(i,i+(s+1))$ and
$(i+(s+1),i)$ components being $-1$, and $0$ elsewhere. Then,
\begin{eqnarray}\label{L2}
\begin{array}{l}
\sum\limits_{i=1}^s\sum\limits_{j=1}^t\frac{t}{2s}
(e_{i,j},e_{i,j+1})E(e_{i,j},e_{i,j+1})^T
=\frac{1}{2}e^T(\frac{t}{s}\sum\limits_{i\in \mathcal{T}}\mathcal{C}_i)e.
\end{array}
\end{eqnarray}
For the third term in (\ref{Riemann Sum}), we have
\begin{eqnarray}\label{L2-2}
\begin{array}{l}
\sum\limits_{i=1}^s\sum\limits_{j=1}^t\frac{\alpha}{2st}
(\frac{1}{2}e_{i,j}^2-\beta)^2
=\frac{\alpha}{8st}\sum\limits_{i=1}^s\sum\limits_{j=1}^te_{i,j}^4
-\frac{\alpha\beta}{2st}\sum\limits_{i=1}^s\sum\limits_{j=1}^te_{i,j}^2
+\frac{\alpha\beta^2}{2}.
\end{array}
\end{eqnarray}
Since
\begin{eqnarray}\label{boundary term}
\begin{array}{l}
\sum\limits_{i=1}^s\sum\limits_{j=1}^te_{i,j}^2
=e^Te-\sum\limits_{i=1}^{s+1}e_{i,t+1}^2-\sum\limits_{j=1}^{t}e_{s+1,j}^2\le\|e\|^2
\end{array}
\end{eqnarray}
and %the two terms $\sum_{i=1}^{n+1}U_{i,m+1}^2$ and
%$\sum_{j=1}^{m}U_{n+1,j}^2$ are constants determined by the boundary
%conditions, we may, without loss of generality, write
%$\sum_{i=1}^n\sum_{j=1}^mu_{i,j}^2 =U^TU=\|U\|^2$ such that
\begin{eqnarray}\label{quadratic product}
\begin{array}{rl}
\sum\limits_{i=1}^s\sum\limits_{j=1}^te_{i,j}^4&
=\sum\limits_{i=1}^s\sum\limits_{j=1}^t(e_{i,j}^2)^2\\
&=(\sum\limits_{i=1}^s\sum\limits_{j=1}^te_{i,j}^2)^2
-\sum\limits_{(i,j)\neq(k,l)}e_{i,j}^2e_{k,l}^2\\
&\leq(\sum\limits_{i=1}^s\sum\limits_{j=1}^te_{i,j}^2)^2\\
%&=(U^TU-\sum_{i=1}^{n+1}U_{i,m+1}^2-\sum_{j=1}^{m}U_{n+1,j}^2)^2\\
%&\leq(U^TU)^2\\
&\le(\|e\|^2)^2,
\end{array}
\end{eqnarray}
the generalized Ginzburg-Landau functional $I^{\alpha}(u)$ in this
example has an estimated upper bound of
\begin{eqnarray}\label{upper bound}
\begin{array}{l}
\frac{\alpha}{8ts}(\|e\|^2)^2+
\frac{1}{2}e^T(\frac{s}{t}\sum\limits_{i\in \mathcal{T}}\mathcal{B}_i+\frac{t}{s}
\sum\limits_{i\in \mathcal{T}}\mathcal{C}_i-\frac{\alpha\beta}{ts}I)e
+\frac{\alpha\beta^2}{2},
\end{array}
\end{eqnarray}
which is of the form (\ref{DWP}) with $x=e$,
$B=(\frac{\alpha}{ts})^{\frac{1}{4}}I$,
$A=(\frac{s}{t}\sum\limits_{i\in
\mathcal{T}}\mathcal{B}_i+\frac{t}{s} \sum\limits_{i\in
\mathcal{T}}\mathcal{C}_i-\frac{\alpha\beta}{ts}I)$, $c=0$, $d=0$,
and $f=0$.

In this paper, we are aimed to categorize all possible
configurations and important features of the double well potential
functions defined by (\ref{DWP}). In part I of the paper, we shall
focus on finding the global minimum solution(s), %to (\ref{DWP}),
deriving the duality theorem, and analyzing the dual of the dual
problem. In part II, we shall study the local (non-global) extremum
solution(s) and prove that for the non-singular case, there is at
most one local non-global minimum point (namely, at most one local
non-global energy well) and at most one local maximum point (at most
one energy barrier). Moreover, the radius of the local maximizer is
always smaller than that of local/global minimizers, which proves
mathematically that the energy barrier (maximizer) is always
surrounded by other energy wells (minimizers). Combining the resutls
from both Part I and Part II, we conclude that, except for some
unbounded cases and singular cases (which can be easily analyzed),
the only non-trivial examples of the double well potential function
in (\ref{DWP}) are those illustrated by Figure 1.

\section{Space reduction and format setting}

Our approach to solving the global minimum solution of (\ref{DWP})
is via the canonical dual transformation, i.e., by introducing
geometrically nonlinear measure (Cauchy-Green type strain) $\xi(x):
R^n \rightarrow R$ defined by
\[
\xi=\frac{1}{2}(Bx-c)^T(Bx-c)-d.
\]
The fourth order polynomial optimization problem (DWP) is then
reduced into the following quadratic program with a single quadratic
equality constraint, called (QP1QC):
\begin{eqnarray}\label{primal0}
\begin{array}{rcl}
    \min && \Pi(x,\xi)=\frac{1}{2} \xi^2+\frac{1}{2}x^TAx-f^Tx\\
    s.t.&&   \xi   =\frac{1}{2}x^TB^TBx-c^TBx+\frac{1}{2}c^Tc-d\\
    &&\ \ x \in R^n, ~~ \xi \in R.
\end{array}
\end{eqnarray}
Notice that there exists a list of research work on solving (QP1QC).
In the rest of this paper, we extend the results of \cite{FLSX,
XFGSZ} to study the  problem (\ref{primal0}) in an explicit manner.

The problem (QP1QC) is a nonconvex optimization problem. It often
requires some dual information for providing a global lower bound in
order to determine the global minimum solution. Lagrange duality is
the most frequently used dual, but it imposes a serious restriction,
called the constraint qualification, on the type of nonconvex
optimization problems to apply. For other problems not even
satisfying any constraint qualification, they are often referred to
as the ``hard case'' in contrast to the easier ones at least with
some dual information to help.

In \cite{FLSX}, for solving a quadratic program with one quadratic
inequality constraint, the (dual) Slater constraint qualification is
relaxed to a more general condition called ``simultaneously
diagonalizable via congruence" (SDC in short). For the (DWP)
problem, the (SDC) condition amounts to the two matrices $A$ and
$B^TB$ are simultaneously diagonalizable via congruence. Namely,
there exists a nonsingular matrix $P$ such that both $P^TAP$ and
$P^TB^TBP$ become diagonal matrices. Unfortunately, for any given
double well potential problem (\ref{DWP}), $A$ and $B^TB$ may not
satisfy (SDC). For example,
$A=\tiny{\left[\begin{array}{cc}1&-1\\-1&0\end{array}\right]}$,
$B=\tiny{\left[\begin{array}{cc}1&-2\\3&-6\end{array}\right]}$ is
such an instance. In other words, some of the double well potential
problem belongs to the ``hard case''. Fortunately, we can show in
this section that $A$ and $B^TB$ of problem (\ref{primal0}) can be
always made to satisfy the (SDC) condition after performing the
following space reduction technique.

%After rearranging terms, the above problem becomes
%\begin{eqnarray}\label{primal}
%\begin{array}{rcl}
%    \min &&\mathrm{P}(t,x)=\frac{1}{2}\xi^2+\frac{1}{2}x^TAx-f^Tx\\
%    s.t.&&(t,x)\in S=\{(t,x)\in R\times R^n|g(t,x)=0\},
%\end{array}
%\end{eqnarray}
%where $g(t,x)=\frac{1}{2}x^TB^TBx-\xi - (B^Tc)^Tx-(d-\frac{1}{2}c^Tc)$,
%and we call (\ref{primal}) the (DWP-QP) problem in short.
%
%Observe that the constraint $g(t,x)$ is positive semi-definite, since we
%try to diagonalize $A$ and $B^TB$ simultaneously via congruence (SDC) to see
%more insights but difficult whether $B^TB$ has at least one zero eigenvalue.
%Hence we reduced the problem (\ref{primal}) into lower dimensional
%which is easily to diagonalize simultaneously via congruence.

Let $U=[u_1,u_2,\cdots,u_r]$ be a basis for the null space of $B$.
First we extend $U$ to a nonsingular $n\times n$ matrix $[U,V]$ such
that $BU=0$ and each $x\in R^n$ can be split as $x=Uy+Vz$ with $y\in
R^r$ and $z\in R^{n-r}$. Then, in terms of variables $\xi$, $y$ and $
z$, the problem (\ref{primal0}) becomes
\begin{align}
&\left\{
\begin{array}{ll}
\displaystyle\min_{\xi, y,z}&\frac{1}{2}\xi^2+\frac{1}{2}(y^T,z^T)
[U,V]^TA[U,V](y^T,z^T)^T-f^T[U,V](y^T,z^T)^T\\
\mbox{s.t.}&
\frac{1}{2}(y^T,z^T)[U,V]^TB^TB[U,V](y^T,z^T)^T
-  \xi -(B^Tc)^T[U,V](y^T,z^T)^T
-(d-\frac{1}{2}c^Tc)=0.
\end{array}
\right.\notag\\
\end{align}
Equivalently,
\begin{align}
 \left\{
\begin{array}{ll}
\displaystyle\min_{\xi, y,z}&\frac{1}{2}\xi^2
+\frac{1}{2}y^TA_{uu}y+\frac{1}{2}z^TA_{vv}z+y^TA_{uv}z-f^TUy-f^TVz\\
\mbox{s.t.}&
\frac{1}{2}z^TB_{vv}z- \xi -(B^Tc)^TVz-(d-\frac{1}{2}c^Tc)=0,
\end{array}
\right.\label{programiny}
\end{align}
where $A_{uu}=U^TAU,A_{vv}=V^TAV,A_{uv}=U^TAV=A_{vu}^T$, $
B_{uu}=U^TB^TBU$,
 $B_{vv}=V^TB^TBV$ and $B_{uv}=U^TB^TBV=B_{vu}^T.$
Notice that the variable splitting ends up with the positive
definiteness of matrix $B_{vv}$ and the elimination  of variable $y$
in  the constraint of (\ref{programiny}). As the result, we can
solve $y$ first with the variables $\xi$ and $z$ fixed. It
amounts to writing (\ref{programiny}) as the following two-level
optimization problem:
\begin{eqnarray}\label{primal002}
\min_{(\xi, z)\in \calE  }
\{\frac{1}{2}\xi^2+\frac{1}{2}z^TA_{vv}z-f^TVz +
\min\limits_{y\in\mathbb{R}^r}\ \{\frac{1}{2}y^TA_{uu}y+y^TA_{uv}z-f^TUy \}\},
\end{eqnarray}
where $\calE  =\{(\xi, z) \in R \times R^{n-r} \vert
\frac{1}{2}z^TB_{vv}z-\xi - (B^Tc)^TVz-(d-\frac{1}{2}c^Tc)=0\}$.

Clearly, if $A_{uu}$ is not positive semi-definite
or if $A_{uu}$ is a zero matrix but $A_{uv}z-U^Tf\neq0$ for some $(\xi,z)\in\mathcal{E}$,
we can immediately conclude that the problems (\ref{primal002}) and  (DWP)
are both unbounded below.
When $A_{uu}=0$ and $A_{uv}z-U^Tf=0,\forall (\xi,z)\in\mathcal{E}$, problem
(\ref{primal002}) is reduced to
\begin{eqnarray}\label{primal003}
&\left\{
\begin{array}{ll}
\displaystyle{\min_{\xi, z}}&\frac{1}{2}\xi^2+\frac{1}{2}z^TA_{vv}z-f^TVz\\
\mbox{s.t.}&\frac{1}{2}z^TB_{vv}z-\xi - (B^Tc)^TVz-(d-\frac{1}{2}c^Tc)=0,
\end{array}
\right.
\end{eqnarray}
which is the format of (\ref{primal0}) with $B_{vv}\succ0$ on a lower dimensional space.

Suppose $A_{uu}\succeq0$ with at least one positive
eigenvalue, and $(\xi, z)\in\cal{E}$. Then, the optimal
solution $y^*$ that solves
\begin{eqnarray}\label{miny}
\min_{y \in R^r}\ \frac{1}{2}y^TA_{uu}y+y^TA_{uv}z-f^TUy
\end{eqnarray}
must satisfy $A_{uu}y+A_{uv}z-U^Tf=0$.
%If $A_{uu}\succ0$, the
%optimal solution to (\ref{miny}),
%\begin{eqnarray}\label{y*case1}
%y^*=-A_{uu}^{-1}U^T(AVz-f),
%\end{eqnarray}
%is unique and the optimal value is
%$$-\frac{1}{2}(AVz-f)^TUA_{uu}^{-1}U^T(AVz-f).$$
%\begin{eqnarray*}
%\begin{array}{l}
%\frac{1}{2}(y^*)^TA_{uu}y^*+(y^*)^TA_{uv}z-f^TUy^*\\
%=\frac{1}{2}(AVz-f)^TUA_{uu}^{-1}U^T(AVz-f)
%-(AVz-f)^TU^TA_{uu}^{-1}A_{uv}z
%+f^TUA_{uu}^{-1}U^T(AVz-f)\\
%=-\frac{1}{2}(AVz-f)^TUA_{uu}^{-1}U^T(AVz-f).
%\end{array}
%\end{eqnarray*}
%Otherwise,
Assume that $W$ is the null space of $A_{uu}$ and $W$ is of
$k$-dimensional. Then, the optimal solution for (\ref{miny}) can be
expressed as
\begin{eqnarray}\label{y*case2}
y^*(\beta)=-A_{uu}^+U^T(AVz-f)+W\beta,\ \forall\beta\in R^k,
\end{eqnarray}
where $A_{uu}^+$ is the Moore-Penrose pseudoinverse of $A_{uu}$, and
$A_{uu}^+=A_{uu}^{-1}$ when $A_{uu}\succ0$.
%$A_{uu}W=0$ and $A_{uu}y+A_{uv}z-U^Tf=0$, we have $W^TU^T(AVz-f)=0$.
%With the similar argument as Case1,
Then the optimal value of (\ref{miny}) becomes
$-\frac{1}{2}(AVz-f)^TUA_{uu}^+U^T(AVz-f).$
%\begin{eqnarray*}
%\begin{array}{l}
%\frac{1}{2}(y^*)^TA_{uu}y^*+(y^*)^TA_{uv}z-f^TUy^*\\
%=-\frac{1}{2}(AVz-f)^TUA_{uu}^+U^T(AVz-f),
%\end{array}
%\end{eqnarray*}
%which is independent to $\beta$.
Consequently, (\ref{primal002}) becomes
\begin{eqnarray}\label{programinz}
\left\{
\begin{array}{ll}
\displaystyle{\min_{\xi, z}}&\frac{1}{2}\xi^2+\frac{1}{2}z^TA_{vv}z-f^TVz
-\frac{1}{2}(AVz-f)^TUA_{uu}^{+}U^T(AVz-f)\\
\mbox{s.t.}&\frac{1}{2}z^TB_{vv}z-\xi - (B^Tc)^TVz-(d-\frac{1}{2}c^Tc)=0.
\end{array}
\right.
\end{eqnarray}
%where $A_{uu}^{+}=A_{uu}^{-1}$ when $A_{uu}$ is positive definite.
Simplifying the expressions, we can write
%
%
%The objective function of (\ref{programinz}) is equivalent to
%%$\begin{array}{l}
%%\frac{1}{2}\xi^2+\frac{1}{2}z^TA_{vv}z-f^TVz
%%-\frac{1}{2}z^TA_{vu}A_{uu}^{-1}A_{uv}z
%%-\frac{1}{2}f^TUA_{uu}^{-1}U^Tf
%%+\frac{1}{2}z^TA_{vu}A_{uu}^{-1}U^Tf
%%+\frac{1}{2}f^TUA_{uu}^{-1}A_{uv}z\\
%%=\frac{1}{2}\xi^2+\frac{1}{2}z^TA_{vv}z
%%-\frac{1}{2}z^TA_{vu}A_{uu}^{-1}A_{uv}z
%%-f^TA_0^TVz
%%-\frac{1}{2}f^TUA_{uu}^{-1}U^Tf\\
%%=
%$$\frac{1}{2}\xi^2+\frac{1}{2}z^TV^T\tilde{A}Vz -\tilde{f}^TVz
%-\frac{1}{2}f^TUA_{uu}^{-1}U^Tf,$$
%%\end{array}$
%where $\tilde{A}=(I-AUA_{uu}^{-1}U^T)A$,
%and $\tilde{f}=(I-AUA_{uu}^{-1}U^T)f$.
%
%Thus (\ref{programinz}) can rewrite as
%\begin{eqnarray}\label{lowdim}
%&\left\{
%\begin{array}{ll}
%\displaystyle{\min_{\xi, z}}&\frac{1}{2}\xi^2+\frac{1}{2}z^TV^T\tilde{A}Vz
%-\tilde{f}^TVz
%-\frac{1}{2}f^TUA_{uu}^{-1}U^Tf\\
%\mbox{s.t.}&\frac{1}{2}z^TB_{vv}z-\xi - (B^Tc)^TVz-(d-\frac{1}{2}c^Tc)=0.
%\end{array}
%\right.
%\end{eqnarray}
%
%Similarly, in Case2,
(\ref{programinz}) as
\begin{eqnarray}\label{lowdim2}
&\left\{
\begin{array}{ll}
\displaystyle{\min_{\xi, z}}&\frac{1}{2}\xi^2+\frac{1}{2}z^TV^T\hat{A}Vz
-\hat{f}^TVz
-\frac{1}{2}f^TUA_{uu}^+U^Tf\\
\mbox{s.t.}&\frac{1}{2}z^TB_{vv}z-\xi - (B^Tc)^TVz-(d-\frac{1}{2}c^Tc)=0,
\end{array}
\right.
\end{eqnarray}
where $\hat{A}=(I-AUA_{uu}^+U^T)A$, $\hat{f}=(I-AUA_{uu}^+U^T)f$,
and $B_{vv}\succ0$. Combining (\ref{primal003}) and (\ref{lowdim2}),
we may simply assume that $B^TB$ is positive definite in (\ref{primal0})
throughout the paper.

Performing Cholesky factorization on the positive definite
$B^TB$ matrix, we have a nonsingular
lower-triangular matrix $P_1$ such that $P_1^T(B^TB)P_1=I$. Since
$A$ and $P_1^T{A}P_1$ are symmetric, there is an orthogonal matrix
$P_2$ such that $P_2^TP_1^T{A}P_1P_2=
Diag(\alpha_1,\alpha_2,\cdots,\alpha_{n})$ is a diagonal matrix and
$P_2^TP_1^T(B^TB)P_1P_2=I$. In other words, ${A}$ and $B^TB$
satisfies the (SDC) condition if $B^TB\succ0$.

Let $P=P_1P_2$ and define $w=P^{-1}x, \psi=P^Tf, \varphi=P^TB^Tc,
\nu=d-\frac{1}{2}c^Tc$. Problem (\ref{primal0}) can be written as
the sum of separated squares in the following form:
\begin{eqnarray}\label{primal}
\begin{array}{lll}
(P)~~~P_0=&\min&\ \Pi(\xi, w)=\frac{1}{2}\xi^2+\sum_{i=1}^{n}(\frac{1}{2}
\alpha_iw_i^2
-\psi_iw_i)\\
&\mbox{s.t.}&(\xi, w)\in \calE_a =\{(\xi, w)\in R\times R^{n}| \;\;
\xi = \Lambda (w)   \},
\end{array}
\end{eqnarray}
where $\Lambda(w)=\sum_{i=1}^{n}(\frac{1}{2}w_i^2-\varphi_iw_i)-  \nu$
is the standard geometrically nonlinear (quadratic in this case) operator in the canonical duality theory.
We shall call (\ref{primal}) the canonical primal problem $(P)$, since it is the
main form that we deal with in this paper.

Let $\sigma\in R$ be the Lagrange multiplier corresponding to the constraint
$\xi-\Lambda(w)=0$ in the canonical
primal problem (\ref{primal}). The Lagrange function becomes
$$L(\xi, w, \sigma)=\Pi(\xi,w)+\sigma(\Lambda(w)-\xi)=\frac{1}{2}\xi^2
+\sum\limits_{i=1}^n(\frac{1}{2}(\alpha_i+\sigma)w_i^2-(\psi_i+\sigma\varphi_i)w_i)-\sigma\xi-\sigma\nu.$$

\noindent When $\sigma \in \calS_a^+ = (\sigma_0,+\infty)$ with
$\sigma_0=\max\{-\alpha_1,-\alpha_2,\cdots,-\alpha_n\}$,
%\{ \sigma \in R | \; \alpha_i+\sigma>0,\ \forall
%i=1,2,\ldots,n \}$,
$L(\xi,w,\sigma)$ is convex in $(\xi,w)$.
The unique global minimum of $L(\xi,w,\sigma)$, denoted by $(\xi(\sigma),w(\sigma))$,
is attained at $w(\sigma)_i=\frac{\psi_i+\sigma\varphi_i}{\alpha_i+\sigma}$ and
$\xi(\sigma)=\sigma$. The dual problem of (P) is thus formulated as
\begin{equation}\label{dual}
(D)~~~\Pi_0^d= \sup_{\sigma\in\calS^+_a} \left\{
\Pi^d(\sigma)= -\frac{1}{2}\sigma^2
-\frac{1}{2}\sum_{i=1}^n\frac{(\psi_i+\sigma\varphi_i)^2}{\alpha_i+\sigma}
-\nu\sigma\right\} .
\end{equation}
It is
clear that this canonical  dual is a concave maximization problem on a convex feasible space
$\calS^+_a$.

% $P_0^d>-\infty$ and  the weak duality \[
% \Pi(w, \xi)\geq \Pi^d(\sigma) \;\; \forall (\xi, w)\in \calE_a, \;\; \forall \sigma\in \calS^+_a
% \] holds.

\begin{Remark}
\em{
In Gao-Strang \cite{gao-strang89}, $L(\xi,w,\sigma)$ is called the
\emph{pseudo-Lagrangian} associated with the
canonical primal problem $(P)$.
Since $L(\xi, w, \sigma)$ is convex in $\xi$ for any given $(w, \sigma)$,
the total complementary function
$\Xi(w, \sigma)$ can be obtained as
\begin{equation}
\Xi(w, \sigma) = \inf_{\xi \in R} L(\xi, w, \sigma) =
- \frac{1}{2}\sigma^2+\sum_{i=1}^{n}(\frac{1}{2}
(\alpha_i+\sigma)w_i^2-(\psi_i+\sigma\varphi_i)w_i) -\sigma\nu.
\end{equation}
By the canonical duality theory \cite{gao-book00},
 the canonical dual function can be defined by
\[
\Pi^d(\sigma) = \inf \{ \Xi(w, \sigma) | \; w \in R^n \},
\]
which is the same as (\ref{dual}) and is called the total
complementary energy.

In finite deformation theory, if the quadratic operator $\Lambda(w) $
represents a Cauchy-Green strain measure and
$\alpha_i = 0 \; \forall i \in \{1, \dots, n\}$, the canonical primal function $\Pi(\xi,w)$ is the
\emph{total potential energy} (see Equation (13) in \cite{gao-strang89}).
The total complementary function $\Xi(w,\sigma)$ then leads to
the well-known \emph{Hellinger-Reissner generalized complementary energy},
and the pseudo-Lagrangian is the \emph{Hu-Washizu generalized potential energy},
proposed independently  by Hu  Hai-Chang \cite{hu-55}  and K. Washizu \cite{wash} in 1955
   (see Chapter 6.3.3 in \cite{gao-book00}).
   The extremality of  these functions and
    the existence of a \emph{total complementary energy}
   $\Pi^d(\sigma)$ as the canonical dual to $\Pi(\xi,w)$
   have been debated in the community of theoretical and applied
    mechanics for several decades  (see \cite{li-gupta}).
    Gao and Strang \cite{gao-strang89} revealed the extremality relations among these functions,
    and the term
    \[
    G_{ap}(w, \sigma) = \sum_{i=1}^{n} \frac{1}{2} (\alpha_i+\sigma)w_i^2
\]
is called the \emph{complementary gap function}.
Their general  global sufficient condition $G_{ap}(w, \sigma) \ge 0, \;\;\forall w \in R^n$
leads to the canonical dual feasible space $\calS^+_a$.
 The result has been generalized to
the cases when $\alpha_i \neq 0, \; \forall i \in \{1, \dots, n\}$ in \cite{gao-amr,gao-opt03}.
The total complementary energy function $\Pi^d(\sigma)$ was first formulated in
nonlinear post-bifurcation analysis \cite{gao-amr},
 where the total potential energy $\Pi(\xi,w)$ is a double-well
functional.  The triality theory proposed in \cite{gao-amr}
can be used to identify both
global and local extrema.}
\end{Remark}

%\proof
%By the Lagrangian function (\ref{Lagrangian}) restricted on
%the region ${\cal F}$ and by the notions $x(\sigma)$ and $\xi(\sigma)$
%in (\ref{x_sigma}) and in (\ref{t_sigma}), we have for any $(\xi, x)\in
%S$ and $\sigma\in{\cal F}$
%$$
%\begin{array}{ll}
%L(\xi, x,\sigma)&\ge\ %\min_{t\in R,x\in R^n}
%{1\over2}\xi(\sigma)^2+{1\over2}x(\sigma)^T(A+\sigma B^TB)x(\sigma)
%-(f+\sigma B^Tc)^Tx(\sigma)-\xi(\sigma)\sigma-(d-{1\over2}c^Tc)\sigma\\
%&=\ -{1\over2}\sigma^2
%-{1\over2}(f+\sigma B^Tc)^T(A+\sigma B^TB)^{-1}(f+\sigma B^Tc)
%-(d-{1\over2}c^Tc)\sigma\\
%&=\Pi^d(\sigma).
%\end{array}
%$$
%Hence, the following weak duality
%The Lagrangian function of the system leading to the following
%inequality
%\begin{eqnarray}\label{Px>=Pd}
%\mathrm{P}(\xi, z)=\sup_{\sigma\in R}L(\xi, z,\sigma)
%\geq\sup_{\sigma\in\mathcal{D}}L(\xi, z,\sigma)\geq \Pi^d_0\geq \Pi^d(\sigma),
%\end{eqnarray}
%where $L(\xi, z,\sigma)=P(\xi, z)+\sigma g^*(\xi, z)$.
%The inequality say that the weak duality holds.
%Furthermore, $\Pi^d_0>-\infty$ since $\mathcal{D}\neq\emptyset$;
%and $\Pi^d<+\infty$ by the fact that $S^*$ in (\ref{primal'}) is non-empty.
%\hfill\eproof

\section{Global minimum solution to the (DWP) problem}\label{sec2}

%Notice that for any $\sigma\in(\sigma_0,+\infty)$, the Lagrange
%function $L_p(\xi, w;\sigma)$ is minimized at $\xi(\sigma)=\sigma$ and
%$w(\sigma)_i=\frac{\psi_i+\sigma\varphi_i}{\alpha_i+\sigma}$.
%According to the saddle point theorem (e.g., Theorem 6.2.5 in
% \cite{Bazaraa}), $(\xi(\sigma),w(\sigma);\sigma)$ is a saddle point
% such that $(\xi(\sigma),w(\sigma))$ and $\sigma$ are optimal,
% respectively, to $(P)$ and $(D)$ if and only if $g(\xi(\sigma),w(\sigma))=0.$

By the equation (\ref{dual}), we have
$\lim\limits_{\sigma\rightarrow+\infty}(-\frac{1}{2}\sigma^2-\nu\sigma)=-\infty$
and
$-\frac{1}{2}\sum_{i=1}^n\frac{(\psi_i+\sigma\varphi_i)^2}{\alpha_i+\sigma}<0$.
Hence $\Pi^d(\sigma)\rightarrow-\infty$ as
$\sigma\rightarrow+\infty$. In other words, the supremum of the dual
problem may occur either at $\sigma^*\in(\sigma_0,+\infty)$, or at
$\sigma^*=\sigma_0$ such that
$\Pi^d_0=\lim_{\sigma\rightarrow\sigma_0^+}{\Pi^d(\sigma)}$.
But the supremum never occurs asymptotically as
$\sigma\rightarrow+\infty$.

If the dual optimal value $\Pi^d_0$ is attained at
$\sigma^*\in(\sigma_0,+\infty)$, it is necessary that
$\frac{d\Pi^d(\sigma^*)}{d\sigma}=0$. Notice that
\begin{eqnarray}\label{diff}
\begin{array}{ll}
    {d\Pi^d(\sigma)\over d\sigma}
    &=\sum\limits_{i=1}^n(\frac{1}{2}w(\sigma)_i^2-\varphi_iw(\sigma)_i)-\xi(\sigma)-\nu\\
    &=g(\xi(\sigma),w(\sigma)),\ \ \forall~~ \sigma\in(\sigma_0,\infty),
\end{array}
\end{eqnarray}
where $g(\xi(\sigma),w(\sigma)) = \Lambda(w(\sigma)) - \xi(\sigma)$, it implies that the vector
$(\xi(\sigma^*),w(\sigma^*),\sigma^*)$ must be a saddle point of $L(\xi,w,\sigma)$ such
that the primal problem $(P)$ is solved by
$(\xi(\sigma^*), w(\sigma^*))$, see \cite{Bazaraa}. In this case, $x^*=Pw(\sigma^*)$ solves
the (DWP) problem with the optimal value
$\frac{1}{2}(\sigma^*)^2+\sum_{i=1}^n(\frac{1}{2}\alpha_i
w(\sigma^*)_i^2-\psi_iw(\sigma^*)_i)$.

Otherwise, the supremum value
$\Pi^d_0=\lim_{\sigma\rightarrow\sigma_0^+}{\Pi^d(\sigma)}>-\infty$ is attained at $\sigma_0$
and $\frac{d\Pi^d(\sigma_0)}{d\sigma}\leq0$.
Let the index set $I=\{i|\alpha_i+\sigma_0=0\}\neq\emptyset$
and $\alpha_j+\sigma_0>0$ for $j\in J=\{1,2,\ldots,n\}\setminus I$.
Since
$\Pi^d_0>-\infty$, we know from (\ref{dual}) that $\psi_i+\sigma_0\varphi_i=0$
for $i\in I$ and
%\begin{center}
\begin{eqnarray}\label{case_2-3}
w(\sigma_0)_i:=\lim\limits_{\sigma\rightarrow\sigma_0^+}w(\sigma)_i=
\lim\limits_{\sigma\rightarrow\sigma_0^+}\frac{\psi_i+\sigma\varphi_i}{\alpha_i+\sigma}
&=&\left\{\begin{array}{ll}\varphi_i,& \textnormal{if}\ i\in I,\\
\frac{\psi_i+\sigma_0\varphi_i}{\alpha_i+\sigma_0},&\textnormal{if}\
i\in J. \end{array}\right.
\end{eqnarray}
The following theorem characterizes the global optimal solution set of ($P$) in this case.

\begin{thm}\label{boundarification}
If the supremum $\Pi^d_0$ is attained when $\sigma$
approaches $\sigma_0$,
then the global optimal solution set of the problem  $(P)$ should satisfy
\begin{eqnarray*}\label{case 2}
& &  w^*_j  =  \frac{\psi_j+\sigma_0\varphi_j}{\alpha_j+\sigma_0},~\mbox{for}~j\in
J,\\
& &
 \sum\limits_{i\in
I}(\frac{1}{2}(w_i^*)^2-\varphi_iw_i^*) =-\sum\limits_{j\in
J}(\frac{1}{2}(w_j^*)^2-\varphi_jw_j^*) +\sigma_0+\nu  .
\end{eqnarray*}
\end{thm}
\proof We first rewrite the problem $(P)$ in terms of the index sets
$I$ and $J$ as follows:
\begin{eqnarray*}%\label{sp-1}
\begin{array}{rl}
\min&\frac{1}{2}\xi^2
+\sum_{i\in I}(\frac{1}{2}\alpha_iw_i^2-\psi_iw_i)
+\sum_{j\in J}(\frac{1}{2}\alpha_jw_j^2-\psi_jw_j)\\
s.t.&\sum_{i\in I}(\frac{1}{2}w_i^2-\varphi_iw_i)
+\sum_{j\in J}(\frac{1}{2}w_j^2-\varphi_jw_j)
-\xi - \nu=0.
\end{array}
\end{eqnarray*}
Since $\alpha_i+\sigma_0=\psi_i+\sigma_0\varphi_i=0,\ \forall i\in
I$, the problem (P) becomes
\begin{eqnarray*}
\begin{array}{rl}
\min&\frac{1}{2}\xi^2
-\sigma_0\sum_{i\in I}(\frac{1}{2}w_i^2-\varphi_iw_i)
+\sum_{j\in J}(\frac{1}{2}\alpha_jw_j^2-\psi_jw_j)\\
s.t.& \sum_{i\in I}(\frac{1}{2}w_i^2-\varphi_iw_i)
=-\sum_{j\in J}(\frac{1}{2}w_j^2-\varphi_jw_j)
+\xi+\nu,
\end{array}
\end{eqnarray*}
which is equivalent to the following unconstrained convex problem
\begin{eqnarray}\label{sp-2}
\begin{array}{rl}
\min\limits_{(\xi, w) \in R\times R^n}&\frac{1}{2}\xi^2
-\sigma_0\xi+\sum\limits_{j\in
J}\left[\frac{1}{2}(\alpha_j+\sigma_0)w_j^2
-(\psi_j+\sigma_0\varphi_j)w_j\right] -\sigma_0\nu,
\end{array}
\end{eqnarray}
because $\alpha_j+\sigma_0>0$ for $j\in J$. Moreover, solving
(\ref{sp-2}) leads to the global optimal solutions of $(P)$ with
those $(\xi^*,w^*)$ such that
\begin{eqnarray}\label{case 2-2}
\xi^*=\sigma_0,\
w^*_j=w(\sigma_0)_j=\frac{\psi_j+\sigma_0\varphi_j}{\alpha_j+\sigma_0},\ j\in J,
\end{eqnarray}
and for $i\in I$,
\begin{eqnarray}\label{case 2-1}
\sum\limits_{i\in I}(\frac{1}{2}(w_i^*)^2-\varphi_iw_i^*)
=-\sum\limits_{j\in J}(\frac{1}{2}(w_j^*)^2-\varphi_jw_j^*)
+\sigma_0+\nu.
\end{eqnarray}
The corresponding optimal value becomes
$$-\frac{1}{2}\sigma_0^2
-\frac{1}{2}\sum_{j\in J}\frac{(\psi_j+\sigma_0\varphi_j)^2}{\alpha_j+\sigma}
-\sigma_0\nu.$$
\hfill\eproof

Suppose $I=\{1,2,3,\ldots,k\}$ and rewrite (\ref{case 2-1}) as
$$\sum\limits_{i=1}^k(w_i^*-\varphi_i)^2
=\sum\limits_{i=1}^k\varphi_i^2
-\sum\limits_{j=k+1}^n[(w_j^*)^2-2\varphi_jw_j^*]+2\sigma_0+2\nu$$
where $w_j,j=k+1,k+2,\ldots,n$ are defined by (\ref{case 2-2}).
In the case when
$\lim\limits_{\sigma\rightarrow\sigma_0^+}\frac{d\Pi^d(\sigma)}{d\sigma}=g(\xi(\sigma_0),w(\sigma_0))<0$,
we have from (\ref{case_2-3}) that
\begin{eqnarray}\label{case 2-4}
0=\sum\limits_{i=1}^k(\varphi_i-\varphi_i)^2
<\sum\limits_{i=1}^k\varphi_i^2
-\sum\limits_{j=k+1}^n[(w_j^*)^2-2\varphi_jw_j^*]+2\sigma_0+2\nu.
\end{eqnarray}
In other words, the optimal solution set is a sphere centered at
$(\varphi_1,\varphi_2,\ldots,\varphi_k)$ with a positive radius of
$\{\sum\limits_{i=1}^k\varphi_i^2
-\sum\limits_{j=k+1}^n[(w_j^*)^2-2\varphi_jw_j^*]+2\sigma_0+2\nu\}^{1/2}$.
On the other hand, when $\lim\limits_{\sigma\rightarrow\sigma_0^+}\frac{d\Pi^d(\sigma)}{d\sigma}=0,$
(\ref{case 2-4}) becomes
\begin{eqnarray}\label{case 2-5}
0=\sum\limits_{i=1}^k\varphi_i^2
-\sum\limits_{j=k+1}^n[(w_j^*)^2-2\varphi_jw_j^*]+2\sigma_0+2\nu,
\end{eqnarray}
which degenerates the optimal solution set of ($P$) to a singleton
since $\sum\limits_{i=1}^k(w_i^*-\varphi_i)^2=0$ forces that
$w_i^*=\varphi_i,\ \forall i=1,2,\ldots,k.$ In the former case
(\ref{case 2-4}),
$w(\sigma_0)=\lim\limits_{\sigma\rightarrow\sigma_0^+}w(\sigma)$ is
not an optimal solution since it locates right at the center of the
sphere. The boundarification technique developed in \cite{FLSX} may
move $w(\sigma_0)$ from the center to the boundary of the sphere
along a null space direction of
$Diag(\alpha_1,\alpha_2,\cdots,\alpha_n)+\sigma_0I$ in order to
solve the primal problem $(P)$. In the latter case (\ref{case 2-5}),
the sphere degenerates to only its center and the optimal solution
of ($P$) is unique which is exactly
$w(\sigma_0)=\lim\limits_{\sigma\rightarrow\sigma_0^+}w(\sigma).$

%%%%%%%%%%%%%%%%%%%%%%%%%%%%%%%%%%%%%%%%%% begin convexification
\section{Dual of the Dual Problem}
By the fact that the canonical dual problem $(D)$ is a concave maximization over a convex
feasible space $\calS^+_a$, the inequality constraints in $\calS^+_a$ can be relaxed by
the traditional Lagrange multiplier method.
In this section, we show that the dual of the dual problem $(D)$ reveals the hidden
convex structure of $(P)$ in (\ref{primal}). This concept of hidden convexity can be referred to
\cite{BT,FLSX}
for different forms of the primal problem.

%will lead to
%a convex relaxation of the nonconvex primal primal problem
%  $(P)$ in (\ref{primal}). This convex relaxation is different with the
%concept of hidden convexity discussed in  \cite{BT} since our  primal problem $(P)$
%is a nonconvex minimization
%with double-well structure. The hidden convexity is a misnomer  concept in this sense.

Writing $(\psi_i+\sigma \varphi_i)^2=\left(\psi_i-\alpha_i\varphi_i
+\varphi_i(\alpha_i+\sigma)\right)^2$,  the problem (D) can be reformulated as
\begin{eqnarray}\label{dual_3}
\begin{array}{lll}P_0^d=&\sup\limits_{\sigma\in R}
&\left\{-\frac{1}{2}\sigma^2-\nu\sigma
-\sum\limits_{i=1}^n(\psi_i-\alpha_i\varphi_i)\varphi_i
-\frac{1}{2}\sum\limits_{i=1}^n\frac{(\psi_i-\alpha_i\varphi_i)^2}
{\alpha_i+\sigma}
-\frac{1}{2}\sum\limits_{i=1}^n\varphi_i^2
(\alpha_i+\sigma)\right\}\\
&\mbox{s.t.}&\alpha_i+\sigma>0,\ i=1,...,n.\end{array}
\end{eqnarray}

\begin{pro}
The Lagrangian dual of Problem (\ref{dual_3}) is the following
linearly constrained convex minimization problem ($P^{dd}$):
\begin{eqnarray}\label{dualofdual}
\begin{array}{rll}
P_0^{dd}=&
\inf\limits_{\lambda \in R^n}&P^{dd}(\lambda)=\sum\limits_{i=1}^n\alpha_i\lambda_i
-\sum\limits_{i=1}^n|\psi_i-\alpha_i\varphi_i|\sqrt{2\lambda_i+\varphi_i^2}
+\frac{1}{2}(\sum\limits_{i=1}^n\lambda_i-\nu)^2-\sum\limits_{i=1}^n(\psi_i-\alpha_i\varphi_i)\varphi_i\\
&\mbox{s.t.}&\lambda_i+\frac{\varphi_i^2}{2}\geq 0,\ i = 1,...,n.
\end{array}
\end{eqnarray}
\end{pro}
\begin{proof}
We first write the (dual) problem (\ref{dual_3}) as
\begin{eqnarray}\label{ri}
\begin{array}{rl}\sup\limits_{\sigma\in R}&\left\{-\frac{1}{2}\sigma^2
-\nu\sigma
-\sum\limits_{i=1}^n(\psi_i-\alpha_i\varphi_i)\varphi_i
-\frac{1}{2}\sum\limits_{i=1}^n\frac{(\psi_i-\alpha_i\varphi_i)^2}
{r_i}-\frac{1}{2}\sum\limits_{i=1}^n\varphi_i^2
r_i\right\}\\
s.t.&\alpha_i+\sigma=r_i,\ i=1,...,n,\\
&r_i>0,\ i=1,...,n.\end{array}
\end{eqnarray}
Let $\lambda_i\in R$ be the Lagragian multipliers associated with
the $i^{th}$ linear equality constraint in (\ref{ri}), then the
Lagrange dual problem becomes
\begin{eqnarray}\label{dualofdual_2}
\begin{array}{l}
-\sum\limits_{i=1}^n(\psi_i-\alpha_i\varphi_i)\varphi_i
+\inf\limits_{\lambda\in R^n}\left\{\sum\limits_{i=1}^n
\alpha_i\lambda_i+h(\lambda)
+\sup\limits_{\sigma\in R}\left[-\frac{1}{2}\sigma^2+\sigma
k(\lambda)\right]\right\}
\end{array}
\end{eqnarray}
where
\begin{eqnarray}
&&h(\lambda)=\sum_{i=1}^n\sup_{r_i>0}\left[-r_i\lambda_i-\frac{\varphi_i^2}{2}r_i
-\frac{(\psi_i-\alpha_i\varphi_i)^2}{2r_i}\right],\label{vz}\\
&&k(\lambda)=\sum_{i=1}^n\lambda_i-\nu\notag.
\end{eqnarray}
The computation of the inner maximization in (\ref{dualofdual_2}) is
$$
\sup_{\sigma\in R}\left(-\frac{1}{2}\sigma^2+\sigma k(\lambda)\right)\\
=\frac{1}{2}k(\lambda)^2.
$$
Consequently, for (\ref{vz}), we have
\begin{eqnarray*}
\sup_{r_i>0}\left[-r_i\lambda_i-\frac{\varphi_i^2}{2}r_i
-\frac{(\psi_i-\alpha_i\varphi_i)^2}{2r_i}\right]
=\left\{ \begin{array}{ll}- |\psi_i
-\alpha_i\varphi_i|\sqrt{2\lambda_i+\varphi_i^2},
& \textnormal{if}\ \lambda_i+\frac{\varphi_i^2}{2}\geq 0,\\+\infty,
&  \textnormal{if}\ \lambda_i+\frac{\varphi_i^2}{2}< 0,
\end{array}  \right.
\end{eqnarray*}
which leads to the result of (\ref{dualofdual}).
\hfill\eproof\end{proof}

To see the correspondence between (\ref{dualofdual}) and $(P)$, we
rewrite $(P)$ by completing the squares as
\begin{eqnarray}\label{SDCDWP}
\begin{array}{ll}
\min\limits_{w}&F(w)=\frac{1}{2}\left\{
\sum\limits_{i=1}^n\left[
\frac{1}{2}(w_i-\varphi_i)^2-\frac{\varphi_i^2}{2}\right]
-\nu\right\}^2\\
&~~~~~~~~~~~~~
+\sum\limits_{i=1}^n \left\{\alpha_i\left[
\frac{1}{2}\left(w_i-\varphi_i\right)^2-\frac{\varphi_i^2}{2}\right]
-\left(\psi_i-\alpha_i\varphi_i\right)w_i\right\}.
\end{array}
\end{eqnarray}
Let $w^*$ be the global minimizer and $i_0\in\{1,...,n\}$ be
arbitrary. Construct $\bar w$ by setting
\begin{center}
$\begin{array}{rcl}
\overline{w}_i&=&\left\{\begin{array}{ll}2\varphi_i-w_i^*& \textnormal{if}\ i=i_0,\\
w_i^* &\textnormal{if}\ i\neq i_0. \end{array}\right.
\end{array}$
\end{center}
Then $F(w^*)\leq F(\bar w)$ and
$(\psi_{i_0}-\alpha_{i_0}\varphi_{i_0})(w_{i_0}^*-\varphi_{i_0})\geq0$.
Since $i_0$ is arbitrarily chosen, it implies that the optimal
solution $w^*$ is also optimal to the following linearly constrained
version:
%
% $(\psi_i-\alpha_i\varphi_i)(w_i^*-\varphi_i)\geq0$ for
%all $i\in[1:n]$. Therefore, the problem (\ref{SDCDWP}) is further
%equivalent to
\begin{eqnarray}\label{SDCDWPconstraint}
%\left\{
\begin{array}{ll}
\min\limits_{w}&\frac{1}{2}\left\{
\sum\limits_{i=1}^n\left[
\frac{1}{2}\left(w_i-\varphi_i\right)^2-\frac{\varphi_i^2}{2}\right]
-\nu\right\}^2
+\sum\limits_{i=1}^n \left\{\alpha_i\left[
\frac{1}{2}\left(w_i-\varphi_i\right)^2-\frac{\varphi_i^2}{2}\right]
-\left(\psi_i-\alpha_i\varphi_i\right)w_i\right\}\\
\mbox{s.t.}&(\psi_i-\alpha_i\varphi_i)(w_i-\varphi_i)\geq0,\ \
i=1,...,n.
\end{array}%\right.
\end{eqnarray}
%\end{thm}

Recall that the problem ($P^{dd}$) in (\ref{dualofdual}) is the
Lagragian dual of the dual problem  and the problem
(\ref{SDCDWPconstraint}) is the original double well problem subject
to $n$ additional linear constraints. We then have the following
result:
%We can see that (\ref{SDCDWPconstraint}) have the same objective function as the primal
%P in (\ref{primal}) with $n$ additional linearly inequality constraints.
%Under these constraints, (\ref{SDCDWPconstraint}) contains the global minimizer $w^*$
%in (P), that is, they have the same optimal solution.
\begin{thm}\label{eqv}
The problem ($P^{dd}$) is of equivalent to the problem
(\ref{SDCDWPconstraint}).
\end{thm}
\proof To prove
(\ref{dualofdual})$\Rightarrow$(\ref{SDCDWPconstraint}), we first
claim that for any $\lambda\in\{\lambda \in
R^n|\lambda_i+\frac{\varphi_i^2}{2}\geq0, i=1,...,n\}$ there exists
$w(\lambda)$ such that $P^{dd}(\lambda)=F(w(\lambda)).$ Let
$\tau_i=\psi_i-\alpha_i\varphi_i,i=1,...,n,$ and define
\begin{eqnarray}\label{nonlineartransformation}
\begin{array}{rcl}
w(\lambda)_i&=&\left\{\begin{array}{ll}\varphi_i+\sqrt{2\lambda_i+\varphi_i^2},
&\textnormal{if}\ \tau_i\geq0;\\
\varphi_i-\sqrt{2\lambda_i+\varphi_i^2}, &\textnormal{if}\ \tau_i<0.
\end{array}\right.
\end{array}
\end{eqnarray}
Then we have $w(\lambda)_i-\varphi_i\geq0$ when $\tau_i\geq0$, and
$w(\lambda)_i-\varphi_i<0$ when $\tau_i<0$. Hence the constraint of
(\ref{SDCDWPconstraint}) is satisfied. Moreover, by
(\ref{nonlineartransformation}), we have
\begin{eqnarray}\label{nonlineartransformation2}
\begin{array}{l}
\lambda_i=\frac{1}{2}(w(\lambda)_i-\varphi_i)^2-\frac{\varphi_i^2}{2},\forall
i=1,...,n.
\end{array}
\end{eqnarray}

The objective of (\ref{dualofdual}) becomes
\begin{eqnarray}\label{PddtoF}
\begin{array}{rl}
P^{dd}(\lambda)
=&\sum\limits_{i=1}^n\alpha_i\lambda_i
-\sum\limits_{i=1}^n|\tau_i|\sqrt{2\lambda_i+\varphi_i^2}
+\frac{1}{2}(\sum\limits_{i=1}^n\lambda_i-\nu)^2-\sum\limits_{i=1}^n\tau_i\varphi_i\\
=&\sum\limits_{i=1}^n\alpha_i\lambda_i
-\sum\limits_{\tau_i>0}\tau_i\sqrt{2\lambda_i+\varphi_i^2}
-\sum\limits_{\tau_i<0}(-\tau_i)\sqrt{2\lambda_i+\varphi_i^2}
+\frac{1}{2}(\sum\limits_{i=1}^n\lambda_i-\nu)^2
-\sum\limits_{i=1}^n\tau_i\varphi_i\\
=&\sum\limits_{i=1}^n\alpha_i\lambda_i
-\sum\limits_{\tau_i>0}\tau_i(w(\lambda)_i-\varphi_i)
-\sum\limits_{\tau_i<0}\tau_i(w(\lambda)_i-\varphi_i)
+\frac{1}{2}(\sum\limits_{i=1}^n\lambda_i-\nu)^2
-\sum\limits_{i=1}^n\tau_i\varphi_i\\
=&\sum\limits_{i=1}^n\alpha_i\lambda_i
-\sum\limits_{i=1}^n\tau_i(w(\lambda)_i-\varphi_i)
+\frac{1}{2}(\sum\limits_{i=1}^n\lambda_i-\nu)^2
-\sum\limits_{i=1}^n\tau_i\varphi_i\\
=&\sum\limits_{i=1}^n\alpha_i[\frac{1}{2}(w(\lambda)_i-\varphi_i)^2-\frac{\varphi_i^2}{2}]
-\sum\limits_{i=1}^n\tau_iw(\lambda)_i
+\frac{1}{2}(\sum\limits_{i=1}^n[\frac{1}{2}(w(\lambda)_i-\varphi_i)^2-\frac{\varphi_i^2}{2}]-\nu)^2\\
=&F(w(\lambda)),
\end{array}
\end{eqnarray}
which is exactly (\ref{SDCDWPconstraint}) subject to $n$ linear
constraints
$(\psi_i-\alpha_i\varphi_i)(w(\lambda)_i-\varphi_i)\geq0,\
i=1,...,n$.

To prove (\ref{SDCDWPconstraint})$\Rightarrow$(\ref{dualofdual}), we
claim that for any $w\in\{w \in
R^n|(\psi_i-\alpha_i\varphi_i)(w_i-\varphi_i)\geq0, i=1,...,n\}$,
there exists $\lambda(w)$ such that $F(w)=P^{dd}(\lambda(w)).$
Define
\begin{eqnarray}\label{nonlineartransformation3}
\begin{array}{l}
\lambda(w)_i=\frac{1}{2}(w_i-\varphi_i)^2-\frac{\varphi_i^2}{2},
i=1,...,n,
\end{array}
\end{eqnarray}
then
$\lambda(w)_i+\frac{\varphi_i^2}{2}=\frac{1}{2}(w_i-\varphi_i)^2\geq0$.
This means that the constraint in (\ref{dualofdual}) always holds.
Moreover, (\ref{nonlineartransformation3}) says that
$(w_i-\varphi_i)^2=2\lambda(w)_i+\varphi_i^2$, with the linearly
constraint in (\ref{SDCDWPconstraint}), we have
\begin{eqnarray}\label{nonlineartransformation4}
\begin{array}{l}
w_i-\varphi_i=
\left\{\begin{array}{ll}\sqrt{2\lambda(w)_i+\varphi_i^2},
&\textnormal{if}\ \tau_i>0;\\
-\sqrt{2\lambda(w)_i+\varphi_i^2},
&\textnormal{if}\
\tau_i<0;\\
\sqrt{2\lambda(w)_i+\varphi_i^2},
&\textnormal{if}\
\tau_i=0\mbox{ and }w_i-\varphi_i\geq0;\\
-\sqrt{2\lambda(w)_i+\varphi_i^2},
&\textnormal{if}\
\tau_i=0\mbox{ and }w_i-\varphi_i<0.
\end{array}\right.
\end{array}
\end{eqnarray}
The objective of (\ref{SDCDWPconstraint}) becomes
\begin{eqnarray*}
\begin{array}{rl}
F(w)
=&\frac{1}{2}\left[\sum\limits_{i=1}^n\lambda(w)_i-\nu\right]^2
+\sum\limits_{i=1}^n\alpha_i\lambda(w)_i
-\sum\limits_{i=1}^n\tau_iw_i\\
=&\frac{1}{2}\left[\sum\limits_{i=1}^n\lambda(w)_i-\nu\right]^2
+\sum\limits_{i=1}^n\alpha_i\lambda(w)_i
-\sum\limits_{i=1}^n\tau_i(w_i-\varphi_i)-\sum\limits_{i=1}^n\tau_i\varphi_i\\
=&\frac{1}{2}\left[\sum\limits_{i=1}^n\lambda(w)_i-\nu\right]^2
+\sum\limits_{i=1}^n\alpha_i\lambda(w)_i
-\sum\limits_{\tau_i>0}\tau_i(w_i-\varphi_i)
-\sum\limits_{\tau_i<0}\tau_i(w_i-\varphi_i)
-\sum\limits_{i=1}^n\tau_i\varphi_i\\
=&\frac{1}{2}\left[\sum\limits_{i=1}^n\lambda(w)_i-\nu\right]^2
+\sum\limits_{i=1}^n\alpha_i\lambda(w)_i
-\sum\limits_{\tau_i>0}|\tau_i|\sqrt{2\lambda(w)_i+\varphi_i^2}\\
&-\sum\limits_{\tau_i<0}(-|\tau_i|)(-\sqrt{2\lambda(w)_i+\varphi_i^2})
-\sum\limits_{i=1}^n\tau_i\varphi_i\\
=&\frac{1}{2}\left[\sum\limits_{i=1}^n\lambda(w)_i-\nu\right]^2
+\sum\limits_{i=1}^n\alpha_i\lambda(w)_i
-\sum\limits_{i=1}^n|\tau_i|\sqrt{2\lambda(w)_i+\varphi_i^2}
-\sum\limits_{i=1}^n\tau_i\varphi_i\\
=&P^{dd}(\lambda(w)),
\end{array}
\end{eqnarray*}
which is exactly (\ref{dualofdual}) subject to $n$ linear
constraints $\lambda(w)_i+\frac{\varphi_i^2}{2}\geq0,\ \ i=1,...,n$.

%By substituting
%$\lambda_i=\frac{1}{2}(w_i-\varphi_i)^2-\frac{{\varphi_i}^2}{2}$ for
%$i\in[1:n]$ into (\ref{SDCDWPconstraint}), or equivalently, by
%
%we can turn (\ref{SDCDWPconstraint}) into
%\begin{eqnarray*}
%\begin{array}{l}
%\left\{
%\begin{array}{ll}
%\min&\frac{1}{2}\left(
%\sum\limits_{i=1}^n\lambda_i-\nu\right)^2
%+\sum\limits_{i=1}^n \left[\alpha_i\lambda_i
%-\left(\psi_i-\alpha_i\varphi_i\right)w_i\right]\\
%\mbox{s.t.}&(\psi_i-\alpha_i\varphi_i)(w_i-\varphi_i)\geq0
%\end{array}\right.\\
%=\left\{
%\begin{array}{ll}
%\min\limits_{\lambda}&\frac{1}{2}\left(
%\sum\limits_{i=1}^n\lambda_i-\nu\right)^2
%+\sum\limits_{i=1}^n \left[\alpha_i\lambda_i
%-\left(\psi_i-\alpha_i\varphi_i\right)\varphi_i
%-\left|\psi_i-\alpha_i\varphi_i\right|\sqrt{2\lambda_i+\varphi_i^2}
%\right]\\
%\mbox{s.t.}&|\psi_i-\alpha_i\varphi_i|\sqrt{2\lambda_i+\varphi_i^2}\geq0,
%\end{array}\right.,\ \forall i\in[1:n],
%\end{array}
%\end{eqnarray*}
%which is indeed (\ref{dualofdual}) if
%$|\psi_i-\alpha_i\varphi_i|\sqrt{2\lambda_i+\varphi_i^2}\geq0$ is
%replaced by $2\lambda_i+\varphi_i^2\geq0$.
\hfill\eproof

Notice that, in Theorem 3 of \cite{FLSX}, it was claimed that the
problem $(P^{dd})$ is equivalent to the primal problem ($P$), and the
nonlinear transformation (\ref{nonlineartransformation}) is
one-to-one. From the above derivations, the correct statements
should be that the dual of the dual problem ($P^{dd}$) is equivalent only to ``part"
of $(P)$ confined by some additional linear constraints.
Indeed, in (\ref{nonlineartransformation}),
if $\tau_i=0$,
we can define $w(\lambda)_i$ as $\varphi_i+\sqrt{2\lambda_i+\varphi_i^2}$
or $\varphi_i-\sqrt{2\lambda_i+\varphi_i^2}$ which leads to the same result.
The nonlinear transformation (\ref{nonlineartransformation}) is not
one-to-one when there is some $i$ such that
$\tau_i=0$.
For each $\lambda$, it may corresponding to at most $2^n$ points of $w(\lambda)$ such that they lead
to the same value of the objective function in (\ref{PddtoF}).
Moreover, in (\ref{nonlineartransformation3}),
for each given $w$, there is exactly one $\lambda(w)$ corresponding to $w$.
This will be shown in Example 3 below.

%
%\begin{eqnarray}\label{nonlineartransformation}
%\begin{array}{ll}
%w_i=\varphi_i\pm\sqrt{2\lambda_i+\varphi_i^2}
%\end{array}
%\end{eqnarray}
%where $(\psi_i-\alpha_i\varphi_i)(w_i-\varphi_i)\geq0,$

%
%we exactly obtain the dual of the dual problem ($P^{dd}$). That is,
%problem (\ref{primal}) is equivalent to ($P^{dd}$) via the nonlinear
%transformation
%\begin{eqnarray}\label{nonlineartransformation}
%\begin{array}{ll}
%w_i=\varphi_i\pm\sqrt{2\lambda_i+\varphi_i^2}
%\end{array}
%\end{eqnarray}
%where $(\psi_i-\alpha_i\varphi_i)(w_i-\varphi_i)\geq0.$
%Moreover, if $(\psi_i-\alpha_i\varphi_i)\neq0$ the nonlinear
%transformation is one-to-one.

%
%Notice that, if $(\psi_i-\alpha_i\varphi_i)\neq0$, the nonlinear
%transformation  (\ref{nonlineartransformation}) is not one-to-one.

%Theorem \ref{eqv} explains why the nonconvex problem (P) has a
%strong duality with no duality gap.

%%%%%%%%%%%%%%%%%%%%%%%%%%%%%%%%%%%%%%%%%%%%%%%% begin example
\section{Numerical Examples}\label{example}

We use some numerical examples to illustrate the (DWP) problem, its
global minimum, and its dual relationship.

\bigskip  \textbf{Example 1}
Let $A=-2, B=(0,-1)^T, c=(0,2)^T, d=14, f=1$. The primal problem (P)
becomes
\begin{eqnarray}\label{ex1primal}
\begin{array}{rl}
min&P(w)=\frac{1}{2}(\frac{1}{2}w^2+2w-12)^2-w^2-w
\end{array}
\end{eqnarray}
The global minimum  locates at $x^*=-7.748$ with the optimal value
$-49.109$.

The dual problem (D) is
\begin{eqnarray}\label{ex1dual}
\begin{array}{rl}
sup&\Pi^d(\sigma)=-\frac{1}{2}\sigma^2
-\frac{(1-2\sigma)^2}{2\sigma-4}-12\sigma\\
s.t.&\sigma\in\mathcal{D}=(2,\infty).
\end{array}
\end{eqnarray}

%\begin{figure}[!hbtp]
%\begin{center}
%\includegraphics[width=1.0\textwidth]{Fig_1.eps}
%\end{center}
%\caption{The graph of $P(w)$ in Example 1 and the corresponding dual of the dual problem.}
%\end{figure}

The supremum occurs at $\sigma^*=2.522\in\mathcal{D}$.
The corresponding primal solution is
$w(\sigma^*)=-7.748$. The dual of the dual problem
(\ref{dualofdual}) is
\begin{eqnarray}\label{ex1dod}
\begin{array}{rl}
P_0^{dd}=&-6+\inf_\lambda[-2\lambda-3\sqrt{2\lambda+4}+\frac{1}{2}(\lambda-12)^2]\\
s.t.&\lambda+2\geq0.
\end{array}
\end{eqnarray}
%whose objective function is drawn at the right of Figure 2. 
The
nonlinear transformation (\ref{nonlineartransformation}) in this
example is $w=-2-\sqrt{2\lambda+4}$, with which we have the dual of
the dual problem:
\begin{eqnarray}\label{ex1dodP}
\begin{array}{ll}
\min&P(w)\\
\mbox{s.t.}&w+2\leq0.
\end{array}
\end{eqnarray}
This is indeed the primal (P) subject to one linear constraint
$w\le-2$. The global minimum of (\ref{ex1dod}) is mapped to the
global minimum of (\ref{ex1dodP}), which is the global minimum of
(P).

\textbf{Example 2} Let
$A=Diag(1,-2),B=\tiny\left[\begin{array}{cc}-0.07&0.04\\-0.01&-1\end{array}\right]\normalsize,
c=(-2,0)^T,d=28,f=(-7,-22)^T.$ After diagonalizing $A$ and $B^TB$
simultaneously, the primal problem (P) has the form
\begin{eqnarray}\label{ex2primal}
\begin{array}{rll}
min&P(w)=&\frac{1}{2}(\frac{1}{2}w_1^2+\frac{1}{2}w_2^2-1.998w_1+0.082w_2-26)^2\\
&&+101.035w_1^2-0.998w_2^2+98.285w_1+21.885w_2
\end{array}
\end{eqnarray}
%The graph of $P(w)$ is depicted at the left of Figure 3. 
Its dual
problem (D) becomes
\begin{eqnarray}\label{ex2dual}
\begin{array}{rl}
sup&\Pi^d(\sigma)=-\frac{1}{2}\sigma^2
-\frac{1}{2}[\frac{(-97.285+1.998\sigma)^2}{\sigma+202.071}+\frac{(-21.885-0.082\sigma)^2}{\sigma-1.997}]-26\sigma\\
s.t.&\sigma\in\mathcal{D}=(1.997,\infty)
\end{array}
\end{eqnarray}

The supremum occurs at
$\sigma^*=4.8475\in\mathcal{D}$. The corresponding
primal solution $w(\sigma^*)=(-0.423,-7.817)^T$ is optimal to (P)
with the optimal value $-243.416$.

%\begin{figure}[!hbtp]
%\begin{center}
%\includegraphics[width=1.0\textwidth]{Fig_2.eps}
%\end{center}
%\caption{The graph of $P(w)$ in Example 2 and the corresponding dual of the dual problem.}
%\end{figure}

The dual of the dual problem (\ref{dualofdual}) has the form
\begin{eqnarray}\label{ex2dod}
\begin{array}{rl}
P_0^{dd}=&999.529
+\inf_\lambda[202.071\lambda_1-1.997\lambda_2
-501.088\sqrt{2\lambda_1+3.993}\\
&-22.049\sqrt{2\lambda_2+0.0067}
+\frac{1}{2}(\lambda_1+\lambda_2-26)^2]\\
s.t.&\lambda_1\geq-1.9967,\ \lambda_2\geq-0.0034.
\end{array}
\end{eqnarray}
Under the one-to-one
nonlinear transformation of $w_1=1.998-\sqrt{2\lambda_1+3.993}$ and
$w_2=-0.082-\sqrt{2\lambda_2+0.0067}$, we have
the primal problem (P) as follows:
\begin{eqnarray}\label{ex2dodP}
\begin{array}{ll}
\min&P(w)\\
\mbox{s.t.}&w_1\leq1.998,\ w_2\leq-0.082.
\end{array}
\end{eqnarray}
%The graph of (\ref{ex2dodP}) is superimposed on Figure 6 with the
%black bold net. 
We can see the optimal solution of
(\ref{ex2dod}) corresponding to $w(\sigma^*)=(-0.423,-7.817)^T$ is
$\lambda^*=(0.9346,29.9117)$ with the same value $-243.416$.

\bigskip \textbf{Example 3 (The Mexican hat)} \noindent Let $A=0_{2\times2},
B=Diag(-0.5,-0.5), c=(0,0)^T, d=38, f=(0,0)^T$. The primal problem
is
\begin{eqnarray}\label{ex3primal}
\begin{array}{rl}
min&P(w)=\frac{1}{2}(\frac{1}{2}w_1^2+\frac{1}{2}w_2^2-38)^2.
\end{array}
\end{eqnarray}
There is a local maximum
at $(0,0)$. The dual problem (D) is
\begin{eqnarray}\label{ex3dual}
\begin{array}{rl}
sup&\Pi^d(\sigma)=-\frac{1}{2}\sigma^2-38\sigma\\
s.t.&\sigma\in\mathcal{D}=(0,\infty).
\end{array}
\end{eqnarray}

The supremum occurs at the left boundary point $\sigma^*=0$, with
$\lim_{\sigma\rightarrow0^+}{d\Pi^d(\sigma)\over
d\sigma}=-38<0$. By Theorem 1, the global minimal solution set is
the circle $S^*=\{(w_1,w_2) \in
R^2|\frac{1}{2}(w_1)^2+\frac{1}{2}(w_2)^2=38\}$, with the optimal
value of $0$.

%\begin{figure}[!hbtp]
%\begin{center}
%\includegraphics[width=1.0\textwidth]{Fig_3.eps}
%\end{center}
%\caption{The graph of $P(w)$ in Example 3 and the corresponding dual of the dual problem.}
%\end{figure}
The dual of the dual in this example is
\begin{eqnarray}\label{ex3dod}
\begin{array}{rl}
P_0^{dd}=&\inf_\lambda\frac{1}{2}(\lambda_1+\lambda_2-38)^2\\
s.t.&\lambda_1\geq0,\ \lambda_2\geq0.
\end{array}
\end{eqnarray}
%whose convex objective function is shown at the right of Figure 4.
Since $(\psi_i-\alpha_i\varphi_i)=0,~i=1,2$, the nonlinear
transformation of $w_i=\pm\sqrt{2\lambda_i},i=1,2$, is not
one-to-one, but it maps (\ref{ex3dod}) back to the {\it entire}
primal problem (\ref{ex3primal}) with {\it no} additional
constraint. The optimal solution set $S^*$ is collapsed into the
line segment $\{(\lambda_1,\lambda_2) \in R^2 \vert
\lambda_1+\lambda_2=38,~\lambda_1\geq0,\ \lambda_2\geq0\}$ in the
dual of the dual problem (\ref{ex3dod}). It is interesting to see
that the local maximum $(0,0)$ in (\ref{ex3primal}) is again mapped
to a local maximum $(0,0)$ in (\ref{ex3dod})

\section{Conclusions of Part I}
To the best of our knowledge, the double well potential problem
proposed in this paper is the first ever mathematical programming
approach to analyze the discrete approximation of the generalized
Ginzburg-Landau functional. The global minimum of the problem can be
obtained by solving the dual of a special type of nonconvex
quadratic minimization problem subject to a single quadratic
equality constraint. After the space reduction, the objective
function and the constraint can be simultaneously diagonalized via
congruence so that the whole problem can be written as the sum of
separated squares. We emphasize that the space reduction also
eliminates the ``hard cases'' in (\ref{DWP}), those that do not
satisfy the Slater constraint qualification and thus fail the dual
approach in general. In the second part of the paper, we go further
to study the analytical properties of the local
minimizers/maximizers of the problem as they also provide
interesting physical and mathematical properties. The results then
lead to an efficient polynomial-time algorithm for computing all
local extremum points, including the local non-global minimizer, the
local maximizer, and the global minimum solution.

\section*{Acknowledgments}
Fang's research work was supported by the US National Science Foundation Grant DMI-0553310.
Gao's research work was supported by AFOSR Grant FA9550-09-1-0285.
Sheu's research work was sponsored
partially by Taiwan NSC 98-2115-M-006 -010 -MY2 and by National
Center for Theoretic Sciences (The southern branch).
 Xing's research work was supported by NSFC No.
11171177.

\end{document}